\documentclass[12pt]{amsart}

\usepackage{amssymb,latexsym,amsmath,extarrows, mathrsfs, amsthm,color,amsrefs}
\usepackage{graphicx}

\usepackage[margin=1in, centering]{geometry}
\usepackage[bookmarksnumbered, colorlinks, plainpages]{hyperref}
\usepackage{hyperref} 
\hypersetup{
    colorlinks=true,       
    linkcolor=blue,          
    citecolor=magenta,        
    filecolor=magenta,      
    urlcolor=cyan           
}

\renewcommand{\theequation}{\arabic{section}.\arabic{equation}}

\usepackage{mathtools}
\mathtoolsset{showonlyrefs,showmanualtags}

\newtheorem{theorem}{Theorem}[section]

\newtheorem{lemma}[theorem]{Lemma}
\newtheorem{proposition}[theorem]{Proposition}
\newtheorem{remark}[theorem]{Remark}

\newtheorem*{definition}{Definition}

\newcommand{\ZFp}{\mathbb{F}_p}

\newcommand{\cA}{{\mathcal A}}
\newcommand{\E}{\mathbb E}
\newcommand{\F}{\mathbb F}

\title[Polynomial Corners]{A Polynomial Roth Theorem for Corners in Finite Fields}

\author[Han]{Rui Han}
\address{Department of Mathematics, Louisiana State University, Baton Rouge LA 70803}
\email{rhan@lsu.edu}
\thanks{RH is supported in part by the National Science Foundation grant DMS-2053285}

\author[Lacey]{Michael T. Lacey}   
\address{School of Mathematics, Georgia Institute of Technology, Atlanta GA 30332, USA}
\email {lacey@math.gatech.edt}
\thanks{MTL: The author is a 2020 Simons Fellow. Research supported in part by grant  from the  National Science Foundation, DMS-1949206}

\author[Yang]{Fan Yang} 
\email{ffyangmath@gmail.com}
\thanks{FY is supported in part by AMS-Simons Travel grant 2019-2021.}

\begin{document}

\date{\today}

\begin{abstract} 
We prove a Roth type theorem for polynomial corners in the finite field setting.  
Let $ \phi _1$ and $ \phi _2$ be two polynomials of distinct degree.  
For sufficiently large primes $ p$, any subset $ A \subset \mathbb F _p \times \mathbb F _p$ with $ \lvert  A\rvert > p ^{2 - \frac1{16}} $ 
contains three points $ (x_1, x_2) ,  (x_1 + \phi _1 (y), x_2),  (x_1, x_2 + \phi _2 (y))$.  
The study of these questions on $ \mathbb F_p$ was started by Bourgain and Chang.  Our Theorem adapts the argument of 
Dong, Li  and Sawin, in particular relying upon deep Weil type inequalities established by N. Katz.  
\end{abstract}

\maketitle

\section{Introduction}

We establish a Roth type theorem for a pair of linearly independent polynomials in the finite field setting. 
But not in $ \mathbb F _p$, rather $ \mathbb F _p \times \mathbb F _p$, with the polynomials acting in different coordinates.  
This we refer to as the corner setting.  

\begin{theorem} \label{thm:Roth}  Let $ p$ be an odd prime.  
Let $\phi_1,\phi_2$ be two linearly independent  polynomials on $\F_p$, degrees not divisible by $ p$, with   $\phi_1(0)=\phi_2(0)=0$. 
Moreover, require both to be quadratic, or have distinct degrees.  
Then any $A\subset \F_p^2 $ with $  \lvert  A \rvert \geq C p ^{2 - \frac{1} {16}} $ contains at least 
 $ C ^{-1}  p^{3 - 3/16} $ triples of the form 
 \begin{equation*}
(x_1,x_2),\;  (x_1+\phi_1(y),x_2) ,\; (x_1, x_2+\phi_2(y)) 
\end{equation*}
 for $ y\in \mathbb F _p$.   The constant $ C = C _{\phi _1, \phi _2}$ is independent of $ p$.  
\end{theorem}

The existence of such triples follows from the multidimensional polynomial {S}zemer\'{e}di Theorem 
of Bergelson and Lieberman \cite{MR1325795}. 
But there are very few prior results in the literature with explicit bounds, and none that we are aware of that are effective. 
Shkredov \cite{MR2223244,MR2266965} addressed the case of the triples $(x_1,x_2),\;  (x_1+ y,x_2) ,\; (x_1, x_2+y) $.   
The best bounds known in this case are double logarithmic, even in the finite field case \cites{MR2532994,MR2289954}.  

There is a small literature on Euclidean analogs of these questions again in the corners setting. 
Shkredov's setting is addressed in \cite{MR3878595}. 
A recent result of Christ, Durcik and Roos \cite{2020arXiv200810140C} in the Euclidean setting, has a corollary that addresses triples of the form 
$(x_1,x_2),\;  (x_1+ y,x_2) ,\; (x_1, x_2+y ^2 ) $. 
A recent closely related paper of Chen, Guo and Li \cite{2020arXiv200813011C} gives a polynomial Roth theorem on $ \mathbb R $.  
One would not expect the Euclidean setting to give the quantitative bounds above.

Our approach follows in the line of investigation started by Bourgain and Chang \cite{MR3704938}, which started the  study of polynomial progressions in $ \mathbb F _p$. We build upon the subsequent work of Peluse \cite{MR3874848} and Dong, Li and Sawin \cite{2017arXiv170900080D}. 
At this point, there is a powerful, and developing, theory of longer progressions, that we will return to below.

The main point is to obtain `smoothing' estimates for averages of functions over polynomial varieties in $ \mathbb F _ p ^2 $. 
We work with functions $ f \;:\; \mathbb F _p ^2 \to \mathbb C $.  We compute expectations 
\begin{equation*}
\mathbb E _{x\in \mathbb F _p ^2 } f  =  \frac1{p ^2 } \sum_{x\in \mathbb F _p ^2 } f (x). 
\end{equation*}
We will also take expectations over $ \mathbb F _p$.  The averages we are interested in are 
\begin{equation} \label{defA} 
\mathcal A (f_1, f_2) := \mathbb E _{y\in \mathbb F _p}  f_1(x_1+\phi_1(y),x_2)f_2(x_1,x_2+\phi_2(y)) 
\end{equation}
 We write $ x= (x_1,x_2) \in \mathbb F _p ^2 $ throughout the paper.   
Above, we take the expectation over $ y\in \mathbb F _p$.   The norms of functions are given by 
\begin{equation*}
\lVert f\rVert_r = \bigl[ \mathbb E _{x} \lvert  f (x)\rvert ^{r}  \bigr] ^{1/r} . 
\end{equation*}

The main inequality compares  $ \mathcal A (f_1, f_2)$ 
to the average of $ f_1$ in the first coordinate, 
times the average of $ f_2$ in the second coordinate.   
These two are close in norm when $ p$ is large.  
That is, in a quantitative sense, the two polynomials act independently of each other.   

\begin{theorem}\label{thm:main}
Let $\phi_1,\phi_2$ be two polynomials on $\F_p$ with distinct degrees, satisfying $\phi_1(0)=\phi_2(0)=0$ . Then the averaging operator $\mathcal{A}$ satisfies
\begin{equation}
\|\mathcal{A}(f_1, f_2)-\E_{x_1}f_1 \cdot  \E_{x_2} f_2 \|_2\lesssim  p^{-1/8}\|f_1\|_4\|f_2\|_4,
\end{equation}
with the implied constant depending only on the degrees of $\phi_1$ and $\phi_2$.
\end{theorem}

\bigskip 

Bourgain and Chang \cite{MR3704938} started the study of polynomial progressions on $ \mathbb F _p$.  
Peluse \cite{MR3874848} and Dong, Li and Sawin \cite{2017arXiv170900080D} extended the work in the setting of progressions of length three, 
as in this paper.   
For longer progressions,  Peluse \cite{MR3934588} established a finite field version of the polynomial {S}zemer\'{e}di Theorem.  
Building on this,  Peluse and Prendiville \cites{2020arXiv200304122P,2019arXiv190302592P,2020arXiv200304121P}  have established logarithmic type bounds for polynomial progressions in $ \mathbb Z$.  
This paper is the first to indicate that similar types of results could be true in the corners setting.

Our argument adapts the argument of Dong, Li and Sawin \cite{2017arXiv170900080D}.  The technique analyzes 
the kernel of the averaging operator in Fourier variables.  Standard considerations lead to expressions that look like Gowers norms of 
the kernel in Fourier variables. The latter are then somewhat complicated exponential sums. 
In the quadratic case, those can be controlled by Gauss sums, and a result of Bombieri \cite{MR200267} on exponential sums with rational arguments. In the general case, one uses the Weil estimates, and deep results of N.~Katz \cite{MR1715519} for `singular' sums of exponentials.

There are some important differences. One should note that our main `smoothing' inequality is weaker than 
\cite{2017arXiv170900080D}, in that it is not an $ L ^{p}$-improving estimate.  
A second important difference is that in the general case, we require different degrees.  
In the one variable setting, this is can be assumed without loss of generality due to a change of variables argument 
\cite{2017arXiv170900080D}*{(3.1)}. This does not seem to be available in the corners setting. 
Nevertheless, in the quadratic case, one can assume linear independence of the two polynomials, and replace the inequalities of 
Katz by those of Bomberi \cite{MR200267} for sums of exponentials along rational functions.

\section{Notation} 

 For a prime $p$,  denote $e_p(x):=e^{2\pi i\frac{x}{p}}$.  
The Fourier transform is defined to be 
\begin{equation*}
\hat{f}(z)=\frac{1}{p^2}\sum_{x\in \ZFp^2} f(x)e_p(-x\cdot z)  = \mathbb E_x f (x) e_p(-x\cdot z). 
\end{equation*}
Many familiar inequalities continue to hold with this notation. 
In particular, Parseval's identity states that 
\begin{equation*}
\lVert f\rVert_2 ^2 = \lVert \hat f \rVert _{\ell ^2 } = \Bigl[ \sum_{z\in \mathbb F _p ^2 } \lvert  \hat f (z) \rvert ^2   \Bigr] ^{1/2} . 
\end{equation*}
 Here, we use the notation $ \lVert  \cdot \rVert_2 $ to denote the norm of $ \mathbb F _p ^2 $, normalized counting measure. 
 And $ \lVert \cdot \rVert _{\ell ^2 }  $ to denote the usual $ \ell ^2 $ norm.  This is done throughout.  
 And, we have Fourier inversion 
 \begin{equation*}
 f (x) = \sum_{z\in \mathbb F _p ^2 } \hat f (z) e_p (x \cdot z). 
\end{equation*}

Many familiar inequalities continue to hold, and we cite them below. One of them is 
for a function $ \phi $ on $ \mathbb F_p$, with Fourier transform $ \widehat \phi (n)= \mathbb E _{x\in \mathbb F_p} \phi (x) e_p (-nx) $, we have 
\begin{equation}\label{e:4}
\lVert \phi\rVert_4 ^{4} = \sum_{\substack{n_1, n_2 ,n_3, n_4 \in \mathbb F _p\\  n_1 - n_2 - n_3 +n_4=0}} 
\widehat \phi (n_1) \overline  {\widehat \phi  (n_2)}\,   \overline  {\widehat \phi (n_3) } \widehat \phi (n_4).  
\end{equation}

\section{Proof of Theorem \ref{thm:Roth}}
We need the following elementary lemma.
\begin{lemma}\label{lem:E3}
Let $f$ be a  function on $\F_p^2$ with $ 0\leq f \leq 1$ . Then, 
\begin{align*}
\E_x ( f\,\E_{x_1}f\,\E_{x_2}f) \gtrsim  (\mathbb E f) ^{3}
\end{align*}
\end{lemma}

\begin{proof}
Let $ \delta = \mathbb E_x f $.  
The function $ g (x_1) =\mathbb E _{x_2} f (x_1, x_2)$ takes values from $ 0 $ to $ 1$, and has $ \mathbb E_{x_1} g (x_1) = \delta $. 
In particular,  letting $ A = \{x_1:\, g \geq \delta /2\}$, we have 
\begin{align*}
\mathbb E_{x_1}  (\mathbf 1_{A }  g) (x_1) \geq \delta/2.
\end{align*}
Indeed, if this were false, we would have $ \mathbb E_{x_1} g(x_1) < \delta $.  
Now, estimate as below, where we insert $ \mathbf 1_{A}$ and use Cauchy--Schwarz.  
\begin{align*}
\E_x ( f\,\E_{x_1}f\,\E_{x_2}f)  &  =  \mathbb E _{\substack{x_1, x_2\\ x_1', x_2'}} 
f (x_1, x_2) f (x_1', x_2) f (x_1, x_2') 
\\
& \geq 
\mathbb E _{\substack{x_1, x_2, x_1'}} 
f (x_1, x_2) f (x_1', x_2)   \cdot  \mathbb E _{x_2'}\mathbf 1_{A} (x_1)  f (x_1, x_2') 
\\
& \geq \tfrac{\delta }2  \mathbb E _{\substack{x_1, x_2, x_1'}}  \mathbf 1_{A} (x_1)
f (x_1, x_2) f (x_1', x_2) 
\\
& \geq \tfrac{\delta }2  \mathbb E _{\substack{x_1, x_2, x_1'}}  \mathbf 1_{A} (x_1) \mathbf 1_{A} (x_1')
f (x_1, x_2) f (x_1', x_2)  
\\
&= \tfrac{\delta }2  \mathbb E _{\substack{  x_2 }}   \bigl(\mathbb E _{x_1} \mathbf 1_{A} (x_1) f (x_1, x_2) ) ^2 
\\
& \geq \tfrac{\delta }2  \bigl( \mathbb E _{\substack{  x_2 }}   \mathbb E _{x_1} \mathbf 1_{A} f ) ^2 \geq \frac{\delta^3 } 8. 
\end{align*}

\end{proof}

\begin{proof}[Proof of Theorem \ref{thm:Roth}] 
Let $f = \mathbf 1_{A}$, and set $ \delta = \mathbb E_x f = p ^{-2} \lvert  A\rvert  \geq  Cp ^{- \frac1{16}}$. 
 Theorem \ref{thm:main} implies that 
\begin{align*}
\E_x \E_y f(x) & f(x_1+\phi_1(y),x_2)   f(x_1,x_2+\phi_2(y))\\
&=\E_x  f  \cdot \mathcal A (f,f)
\\
&\geq \frac{ \delta^3}{8}- Cp^{-\frac{1}{8}}\|f\|_2\|f\|_4^2\\
&= \frac{\delta^3}{8}- C p^{-\frac{1}{8}}\delta . 
\end{align*}
But under the conditions of the Theorem,  the last bound is at least $ C ^{-1} \delta ^{3} \gtrsim   p ^{- \frac3{16}} $.    
The polynomial triples in $ A$ are indexed by $ x=(x_1, x_2) \in \mathbb F _p ^2 $ and $ y\in \mathbb F _p$, so that the 
conclusion follows.  
\end{proof}

\section{Proof of Theorem \ref{thm:main}} \label{section: proof}
Expand  $f_1$ and $f_2$ in Fourier variables 
\begin{align}
\cA (f_1, f_2)(x) 
& =  \sum_{m,n }  \hat f_1 (n) \hat f_2 (m)  e_p ( (m+n)x ) \mathbb E _y  e _{p} (n_1\phi _1 (y) + m_2\phi _2 (y))
\\ & = \sum_{m,n }  \hat f_1 (n) \hat f_2 (m)  e_p ( (m+n)x ) K (n_1,m_2) , 
\\
\textup{where} \quad K (n_1,m_2) & =  \mathbb E _y  e _{p} (n_1\phi _1 (y) + m_2\phi _2 (y)) .  
\end{align}
The kernel $ K (n_1,m_2)$ given above plays the decisive role.  
Let us observe that linear independence of $ \phi _1$ and $ \phi _2$ together with the fundamental work of Weil \cite{MR4242}, 
imply that 
\begin{equation}\label{e:Weil}
\lvert  K (n_1, m_2 )\rvert  = \lvert  \mathbb E _y  e _{p} (n_1\phi _1 (y) + m_2\phi _2 (y)) \rvert  \lesssim  \frac1{ \sqrt{p}    }, \qquad  (n_1, m_2) \neq 0. 
\end{equation}

The sum over $ m,n \in \mathbb F _p ^2 $ is divided into the collections 
\begin{gather*}
 \mathcal J_1 = \{  (n,m) \;:\;   n_1 = m_2 =0\} , 
 \\
\mathcal J_2 = \{  (n,m) \;:\;  n_1 \neq 0,\ m_2=0\}, 
 \\
\mathcal J_3 = \mathbb F_p \times \mathbb F _p \setminus  (\mathcal J_1 \cup \mathcal J_2) . 
\end{gather*}
This gives this three sums. 
\begin{align}
J_1 &=\sum_{n_2}\sum_{m_1}\hat{f}_1(0,n_2)\hat{f}_2(m_1,0)e_p(m_1x_1+n_2x_2)
\\ \label{e:J1}
&=\mathbb{E}_{x_1} f_1 \cdot \mathbb{E}_{x_2} f_2, 
\\  \\ 
J_2&=\sum_{n_1\neq 0}\sum_{n_2}\sum_{m_1}\hat{f}_1(n) \hat{f}_2(m_1,0) K(n_1,0) e_p(n\cdot x+m_1x_1)
\\    \label{e:J2}
&=\sum_{n_2}\sum_{m_1}\sum_{n_1\neq 0} \hat{f}_1(n)\hat{f}_2(m_1-n_1,0) K(n_1,0) e_p(m_1x_1+n_2x_2) , 
\\  \\ 
J_3&=\sum_{n} \sum_{m_1}\sum_{m_2\neq 0} \hat{f}_1(n) \hat{f}_2(m) K(n_1,m_2) e_p((n+m)\cdot x)
\\
&=\sum_{n}\sum_{m_1}\sum_{m_2\neq 0}\hat{f}_1(n-m)\hat{f}_2(m) K(n_1-m_1,m_2) e_p(n\cdot x)
\\  \label{e:J3}
&=\sum_{n}\sum_{m} \hat{f}_1(n-m)\hat{f}_2(m) \tilde{K}(n_1-m_1,m_2) e_p(n\cdot x), 
\end{align} 
where in the last line we modify the definition of $ K $ from \eqref{e:Weil} to 
\begin{align*}
\tilde{K}(m):=
\begin{cases}
K(m_1, m_2),  &  \text{if } m_2\neq 0\\
0  &  \text{if } m_2=0
\end{cases}
\end{align*}

The term $ J_1$ in \eqref{e:J1} is the one we subtract off.  We  estimate the remaining two terms.  
The estimate for $ J_2$ is the straight forward one.   It is addressed in Lemma~\ref{lem:J2}, while the 
sophisticated term  $ J_3$ in \eqref{e:J3} is addressed in Lemma~\ref{lem:J3}. 
From these two Lemmas, we have 
\begin{align}
\lVert \cA_\Gamma(f_1, f_2) - 
\mathbb{E}_{x_1} f_1 \cdot \mathbb{E}_{x_2} f_2 \rVert_2  & \leq \lVert  J_2\rVert_2 + \lVert  J_3\rVert_2 
\\
& \lesssim p ^{- \frac{1}8} \lVert f_1\rVert_4 \lVert f_2\rVert_4. 
\end{align}
That is the conclusion of the  Theorem.   And we turn to the Lemmas.

\begin{lemma}\label{lem:J2}
For $J_2$, we have the following estimate.
\begin{align*}
\|J_2\|_2  \lesssim  p^{-\frac{1}{4}} \|f_1\|_4 \|f_2\|_4.
\end{align*}
\end{lemma}

\begin{proof}
The condition that $ n_1 \neq 0 $ means that Weil's inequality \eqref{e:Weil} holds.  
From the line \eqref{e:J2}, apply  Parseval's identity in the variables $ m_1$ and $ n_2$.   Then square out the norm. 
\begin{align*}
\|J_2\|_2^2&=
\Bigl\lVert \sum_{n_1\neq 0} \hat{f}_1(n)\hat{f}_2(m_1-n_1,0) K(n_1,0)\Bigr\rVert _{\ell^2_{m_1,n_2}}^2 
\\
&= \sum_{n_2 }\Bigl\lVert \sum_{n_1\neq 0} \hat{f}_1(n)\hat{f}_2(m_1-n_1,0) K(n_1,0)\Bigr\rVert _{\ell^2_{m_1}}^2 
\\
&= \sum_{ \substack{ m_1, n_2 \\ n_1, n_1'\neq 0}}  \hat{f}_1(n_1, n_2)\hat{f}_2(m_1-n_1,0) K(n_1,0)
\overline { \hat{f}_1(n_1',n_2)\hat{f}_2(m_1-n_1',0) K(n_1',0)}
\\
&= \sum_{ \substack{ m_1, n_2 \\ n_1 \neq 0,\  u\neq -n_1}}  \hat{f}_1(n_1, n_2)\hat{f}_2(m_1-n_1,0) K(n_1,0)
 \\ & \qquad  \qquad \times \overline { \hat{f}_1(n_1+u,n_2)\hat{f}_2(m_1-n_1-u,0) K(n_1+u,0)}
\end{align*}
Having squared out the $ \ell ^{2}$ norm, we set $ n_1' = n_1 +u$. 
 The last line is rewritten using the  notation 
\begin{equation}\label{e:Delta}
\Delta _{ u} \phi (n) = \phi (n) \overline  {\phi  (n +u) }, \qquad n,u\in \mathbb F _p ^2 . 
\end{equation}
We have 
\begin{align}
\|J_2\|_2^2  & = \sum_{m_1}\
\sum_{ \substack{n \;:\; n_1 \neq 0}}\ \sum_{ u  \;:\;  u\neq -n_1}  
\Delta _{(u,0)}  \hat{f}_1(n_1, n_2) \Delta _{(-u,0)} \hat{f}_2(m_1-n_1,0)   \Delta _{(u,0)}K(n_1,0) 
\\
\noalign{\noindent above, we can change variables, sending $m_1 $ to $ m_1 + n_1$,}
& = 
\sum_{ \substack{n \;:\; n_1 \neq 0}}\ \sum_{ u  \;:\;  u\neq -n_1}  \ \sum_{m_1}
\Delta _{(u,0)}  \hat{f}_1(n_1, n_2) \Delta _{(-u,0)} \hat{f}_2(m_1,0)   \Delta _{(u,0)}K(n_1,0) 
\\  \label{e:TwoNorms}
&\leq 
\Bigl\lVert   \sum_{m_1} \Delta _{(-u,0)} \hat{f}_2(m_1,0)    \Bigr\rVert _{\ell ^2_u} 
\cdot 
\Bigl\lVert 
\sum_{ \substack{n \;:\; n_1 \neq 0,  -u }}  
\Delta _{(u,0)}  \hat{f}_1(n_1, n_2)    \Delta _{(u,0)}K(n_1,0) 
\Bigr\rVert _{\ell ^{2}_u}. 
\end{align}

We estimate the two norm in \eqref{e:TwoNorms}. On the one hand, squaring out the norm below and appealing to \eqref{e:4}, we have 
\begin{align}
\Bigl\lVert   \sum_{m_1} \Delta _{(-u,0)} \hat{f}_2(m_1,0)    \Bigr\rVert _{\ell ^2_u}  ^2 
& = 
\sum_{u}  \sum_{m_1 , m_1'}
 \hat{f}_2(m_1,0)  \overline  { \hat{f}_2(m_1 -u,0) }  \,  \overline { \hat{f}_2(m_1',0) }      { \hat{f}_2(m_1 '-u,0) }   
 \\  \label{e:44}
 & = \lVert \mathbb E _{x_2} f_2 \rVert _{4} ^{4}  \leq \lVert f\rVert _{4} ^{4}.  
\end{align}

For the second norm in \eqref{e:TwoNorms}, the summing condition on $ n_1 \neq 0, -u$ means that the Weil estimate \eqref{e:Weil} holds, 
giving us $ \lvert  \Delta _{(u,0)}K(n_1,0)  \rvert \lesssim 1/p $. Thus, 
\begin{align}
\Bigl\lVert 
\sum_{ \substack{n \;:\; n_1 \neq 0,  -u }}  
\Delta _{(u,0)}  \hat{f}_1(n_1, n_2)    \Delta _{(u,0)}K(n_1,0) 
\Bigr\rVert _{\ell ^{2}_u} 
& \lesssim 
\frac{1}p 
\biggl\lVert  
\sum_{ \substack{ n_1 \;:\;n_1 \neq 0,  -u }}  
\biggl\lvert   
\sum_{n_2}
\Delta _{(u,0)}  \hat{f}_1(n_1, n_2)   \biggr\rvert 
\biggr\rVert _{\ell ^{2}_u}
\\  \label{e:J21}
& \lesssim \frac{1} {\sqrt{p}    } 
\Bigl\lVert 
\sum_{n_2}   \Delta _{(u,0)}  \hat{f}_1(n_1, n_2)     
\Bigr\rVert _{\ell ^{2}_{n_1,u}}
\end{align}
We continue with this last norm, squaring  it out. 
\begin{align}
\Bigl\lVert 
\sum_{n_2}   \Delta _{(u,0)}  \hat{f}_1(n_1, n_2)     
\Bigr\rVert _{\ell ^{2}_{n_1,u}}^2
 & = 
\sum_{ \substack{n_1, n_2\\ n_2' , u  }}  
  \Delta _{(u,0)}  \hat{f}_1(n_1, n_2)    \overline {  \Delta _{(u,0)}  \hat{f}_1(n_1, n_2')   }  
  \\
  & = 
  \sum_{ \substack{n_1, n_2\\ n_2' , u  }}  
   \hat{f}_1(n_1, n_2)   \overline {\hat{f}_1(n_1 +u, n_2)  } \,   \overline {  \hat{f}_1(n_1, n_2')   }   
    \hat{f}_1(n_1 +u , n_2') 
    \\ &= 
  \sum_{ \substack{n_1, n_2\\ n_1', n_2'}}  
   \hat{f}_1(n_1, n_2)   \overline {\hat{f}_1(n_1', n_2)  } \,   \overline {  \hat{f}_1(n_1, n_2')   }   
    \hat{f}_1(n_1' , n_2')  
    \\ &=   \label{e:J22}
    \mathbb E _{\substack{x_1, x_2\\ x_1', x_2' }} 
    f (x_1, x_2) f (x_1', x_2)     f (x_1, x_2') f (x_1', x_2') \leq \lVert f\rVert_4 ^{4}.  
\end{align}
Combining \eqref{e:44}, \eqref{e:J21} and \eqref{e:J22} completes the proof.  
\end{proof}

Now we turn to the more sophisticated  estimates of $J_3$.   The following Lemma with Lemma~\ref{lem:J2} completes the proof of Theorem~\ref{thm:main}.  

\begin{lemma}\label{lem:J3}
We have 
\begin{align*}
\|J_3\|_2 \lesssim  p^{-\frac{1}{8}}\|f_1\|_4 \|f_2\|_4.
\end{align*}
\end{lemma}
\begin{proof}
From the equality \eqref{e:J3}, apply     Parseval's identity in the variable $ n \in \mathbb F _p ^2 $. 
\begin{align}
\|J_3\|_2^2
&=\Bigl\lVert 
\sum_m\hat{f}_1(n-m)\hat{f}_2(m) \tilde{K}(n_1-m_1,m_2) 
\Bigr\rVert_{\ell_n^2}^2   \notag 
\\
&=\sum_n\sum_{m,m'} \hat{f}_1(n-m)\overline{\hat{f}_1(n-m')}\hat{f}_2(m)\overline{\hat{f}_2(m')} 
\tilde{K}(n_1-m_1,m_2)\overline{\tilde{K}(n_1-m_1',m_2')}  \notag 
\\ 
\noalign{\noindent letting $m'=m+h$ and using the notation  $ \Delta _{h} \phi $ defined in  \eqref{e:Delta},  }
&=\sum_n\sum_{m,h}(\Delta_{-h}\hat{f}_1)(n-m)(\Delta_h\hat{f}_2)(m) 
\Delta_{(-h_1,h_2)}\tilde{K}(n_1-m_1,m_2)
\notag \\
&=\sum_{n,m,h}\Delta_{-h}\hat{f}_1(n) \Delta_h\hat{f}_2(m)
\Delta_{(-h_1,h_2)}\tilde{K}(n_1,m_2)
\label{eq:J3_1}
\\
&=\sum_{h} \sum_{n_1,m_2}\left(\sum_{n_2}(\Delta_{-h}\hat{f}_1)(n)\right) \left(\sum_{m_1}(\Delta_h\hat{f}_2)(m) \right)
(\Delta_{(-h_1,h_2)}\tilde{K})(n_1,m_2)
\\
& =:\sum_{h} I(h).\label{eq:J3_2}
\end{align}

Now we estimate each $I(h)$ in \eqref{eq:J3_2}.
When $h=0$, we have
\begin{align}\label{eq:J3_h0}
I(0)=\sum_{n,m} |\hat{f}_1(n)|^2 |\hat{f}_2(m)|^2 |\tilde{K}(n_1,m_2)|^2 \leq p^{-1} \|f_1\|_2^2 \|f_2\|_2^2.
\end{align}
In this case, the  Weil estimate \eqref{e:Weil} applies, since $ m_2 \neq 0$.

For $ h\neq 0$, we need  Lemma \ref{lem:J3sub}, which is the consequence of a deep extension of Weil's estimates due to Katz \cite{MR1715519}.
Lemma~\ref{lem:J3sub} implies that 
\begin{align*}
I(h)\lesssim p^{-\frac{1}{4}} \left\|\sum_{n_2}\Delta_{-h}\hat{f}_1(n)\right\|_{\ell_{n_1}^2} \left\|\sum_{m_1}\Delta_h\hat{f}_2(m) \right\|_{\ell_{m_2}^2}.
\end{align*}
Next, by Cauchy-Schwarz in $h$, we have
\begin{align}\label{eq:J3_hn0}
\sum_{h\neq 0} I(h)\lesssim p^{-\frac{1}{4}} \left\|\sum_{n_2}\Delta_{-h}\hat{f}_1(n)\right\|_{\ell_{n_1,h}^2} \left\|\sum_{m_1}\Delta_h\hat{f}_2(m) \right\|_{\ell_{m_2,h}^2}
\end{align}
Concerning the two terms on the right,  their estimates are the same by symmetry. For the first term, we have 
\begin{align}\label{eq:J3_4}
\Bigl\lVert 
\sum_{n_2}\Delta_{-h}\hat{f}(n) 
\Bigr\rVert_{\ell_{n_1,h}^2}\leq \|f\|_4^2.
\end{align} 
Indeed, the left side above, squared out is 
\begin{align*}
\Bigl\lVert 
\sum_{n_2}\Delta_{-h}\hat{f}(n) 
\Bigr\rVert_{\ell_{n_1,h}^2} ^2 
 &= \sum_{n_1,h,n_2,n_2'} \hat{f}(n_1,n_2)\overline{\hat{f}(n_1-h_1,n_2-h_2)}\overline{\hat{f}(n_1,n_2')}\hat{f}(n_1-h_1,n_2'-h_2)
 \\ 
  &= \sum_{n_1,n_1',n_2,n_2', n} \hat{f}(n_1,n_2)\overline{\hat{f}(n_1',n_2 - h_2)}\overline{\hat{f}(n_1,n_2')}\hat{f}(n_1',n_2' - h_2) 
  \\ &=\E_{x_2} (\E_{x_1}f^2(x_1,x_2))^2\leq \|f\|_4^4.
 \end{align*}
The last equality is easy to check directly.  

Plugging the estimates \eqref{eq:J3_h0} and \eqref{eq:J3_hn0} into \eqref{eq:J3_2}, we have
\begin{align}\label{eq:J3_3}
\|J_3\|_2^2
&  \lesssim p^{-1}\|f_1\|_2^2\|f_2\|_2^2+p^{-\frac{1}{4}}  \lVert f_1\rVert_4 ^2 \lVert f_2\rVert_4 ^2 . 
\end{align}

This proves Lemma \ref{lem:J3}. 
\end{proof}

\begin{lemma}\label{lem:J3sub}  For any $ F_1, F_2$ functions on $ \{0 ,\dotsc, p-1\}$ and $ h \neq 0 \in \mathbb F _p ^2 $, we have 
\begin{align*}
\sum_{n_1,m_2} F_1(n_1) F_2(m_2)(\Delta_{(h_1,h_2)}\tilde{K})(n_1,m_2)
\lesssim  p^{-\frac{1}{4}} \|F_1\|_{\ell_{n_1}^2} \|F_2\|_{\ell_{m_2}^2} .
\end{align*}
\end{lemma}
\begin{proof}[Proof of Lemma \ref{lem:J3sub}]
By duality, this is equivalent to 
\begin{align*}
\|\sum_{m_2} F_2(m_2)(\Delta_{(h_1,h_2)}\tilde{K})(n_1,m_2)\|_{\ell_{n_1}^2}^2\leq p^{-\frac{1}{2}} \|F_2(m_2)\|_{\ell_{m_2}^2}^2 .
\end{align*}
Expanding the left hand side, we have
\begin{align*}
\mathcal{B}:=\sum_{m_2,m_2'} F_2(m_2)\overline{F_2(m_2')} 
\sum_{n_1}(\Delta_{(h_1,h_2)}\tilde{K})(n_1,m_2)
\overline{(\Delta_{(h_1,h_2)}\tilde{K})(n_1,m_2')}.
\end{align*}
Here we use the following lemma.  To state it, we need this definition from  \cite{2017arXiv170900080D}. 

\begin{definition}
A set $D\subset \F_p^2$ is called a \emph{generalized diagonal} if for any $x\in \ZFp$, there are $O(1)$ $y$'s such that $(x,y)\in D$ and for any $y\in\ZFp$ there are $O(1)$ $x$'s such
that $(x,y)\in D$. The implied constant must be independent of $p$.
\end{definition}

\begin{lemma}\label{lem:K4}
Assume $\phi_1,\phi_2$ have distinct degrees, or both be quadratic.  
Then there exists a generalized diagonal set $D_h\in \F_p^2$ such that 
for $(m_2,m_2')\notin D_h$, we have
\begin{align}  \label{eq: final exp sum estimate}
\left|\sum_{n_1}(\Delta_{(h_1,h_2)}\tilde{K})(n_1,m_2)
\overline{(\Delta_{(h_1,h_2)}\tilde{K})(n_1,m_2')}\right|\lesssim p^{-\frac{3}{2}}.
\end{align}
\end{lemma}
\begin{remark}
This lemma is essentially from \cite{2017arXiv170900080D}*{Thm 3.1}, we include its proof in the appendix. 
This is the only place where we make the assumption that $\phi_1, \phi_2$ have distinct degrees, if the 
polynomials are not quadratic.  
Other parts of the proof only require  $\phi_1,\phi_2$ be linearly independent. 
In \cite{2017arXiv170900080D}, such a distinct degree assumption was not needed for the proof, as an 
early step of the argument \cite{2017arXiv170900080D}*{(3.1)}  allowed one to reduce to the case of distinct degrees.  
We return to this point in the appendix.  
\end{remark}
The lemma above gives us estimate when $(m_2,m_2')\notin D_h$. When $(m_2,m_2')\in D_h$, we use \eqref{e:Weil} which implies
\begin{align*}
\left|\sum_{n_1}(\Delta_{(h_1,h_2)}\tilde{K})(n_1,m_2)
\overline{(\Delta_{(h_1,h_2)}\tilde{K})(n_1,m_2')}\right|\lesssim p^{-1}.
\end{align*}
With these estimates in hand, we have
\begin{align*}
\mathcal{B}\lesssim \sum_{m_2,m_2'\in D_h} p^{-1}|F_2(m_2)||F_2(m_2')|+\sum_{m_2,m_2'\notin D_h}p^{-\frac{3}{2}}|F_2(m_2)||F_2(m_2')|\lesssim p^{-\frac{1}{2}}\|F_2\|_{\ell_{m_2}^2}^2,
\end{align*}
by Cauchy-Schwartz inequality. This proves Lemma \ref{lem:J3sub}.
\end{proof}

\appendix 
\setcounter{equation}{0}
\renewcommand{\theequation}{A.\arabic{equation}}

\section{Proof of Lemma \ref{lem:K4}}
Now we compute the following when $m_2\neq 0, m_2'\neq 0, m_2+h_2\neq 0, m_2'+h_2\neq 0$.

\begin{align}
\sum_{n_1} & (\Delta_{(-h_1,h_2)}\tilde{K})(n_1,m_2)
\overline{(\Delta_{(-h_1,h_2)}\tilde{K})(n_1,m_2')}\\
&=\sum_{n_1}\E_{y_1}\E_{y_2}\E_{y_3}\E_{y_4} e_p(n_1 G(y_1,y_2,y_3,y_4)+H(y_1,y_2,y_3,y_4))\\
&=\frac{1}{p^3}\sum_{\substack{y_1,y_2,y_3,y_4\\G(y_1,y_2,y_3,y_4)=0}}e_p(H(y_1,y_2,y_3,y_4)),   \label{e:GH}
\end{align}
where
\begin{align}
\begin{cases}
G(y_1,y_2,y_3,y_4):=\phi_1(y_1)-\phi_1(y_2)-\phi_1(y_3)+\phi_1(y_4)\\
H(y_1,y_2,y_3,y_4):=h_1(\phi_1(y_2)-\phi_1(y_4))+m_2(\phi_2(y_1)-\phi_2(y_2))\\
\qquad\qquad +m_2'(\phi_2(y_4)-\phi_2(y_3)) +h_2(\phi_2(y_4)-\phi_2(y_2))
\end{cases}
\end{align}
That is, in \eqref{e:GH}, we are summing  over points 
determined by the zero locus of  the polynomial $ G$,  with the exponential of values of the polynomial $ H$.

 Katz has generalized Deligne's theorem to exponential sums over smooth affine varieties \cite{MR617009}, and     singular algebraic varieties \cite{MR1715519}. We need the following special case of \cite{MR1715519}*{Theorem 4}.  (Here, we quote \cite{2017arXiv170900080D}: `The reader could skip its long proof and use it as a ``black box'' on an early reading of the paper.') 
\begin{theorem} \label{thm: Katz}
Let $G,H\in \ZFp[X_1,\dots,X_4]$. Assume that the degree of $H$ is indivisible by $p$, the homogeneous leading term of $G$ defines a smooth projective hypersurface, and the homogeneous leading terms of $G$ and that of $H$ together define a smooth co-dimension $2$ variety in the projective space. Then the following holds
\begin{equation*}
\sum_{\substack{y_1,y_2,y_3,y_4\\G(y_1,y_2,y_3,y_4)=0}}e_p(H(y_1,y_2,y_3,y_4))\lesssim p^{\frac{3}{2}}.
\end{equation*}

\end{theorem}
Now we are ready to prove \eqref{eq: final exp sum estimate}.  
We verify that the expression in \eqref{e:GH} satisfies the hypotheses of Theorem \ref{thm: Katz}.  
The first two conditions in the theorem are easy to check. We elaborate on the third condition, namely the `smooth co-dimension $2$ variety' 
condition.
It is split into  two cases separately: $d_1<d_2$ and $d_1=d_2$.

First assume $d_1<d_2$. Let $az^{d_1}$ and $bz^{d_2}$ denote the leading term of $\phi_1$ and $\phi_2$, resp. The homogeneous leading term of $G$ and $H$ are given below
\begin{gather*}
G_{d_1}(y_1,y_2,y_3,y_4):=ay_1^{d_1}-ay_2^{d_1}-ay_3^{d_1}+ay_4^{d_1},
\\
H_{d_2}(y_1,y_2,y_3,y_4):=bm_2 y_1^{d_2}-b(m_2+h_2)y_2^{d_2}-bm_2'y_3^{d_2}+b(m_2'+h_2)y_4^{d_2}. 
\end{gather*}

The Jacobian matrix is 
\begin{equation*}
J:=
\begin{bmatrix}
\nabla G_{d_1}\\ \nabla H_{d_2}
\end{bmatrix}=
\begin{bmatrix}
d_1ay_1^{d_1-1} & -d_1ay_2^{d_1-1} & -d_1ay_3^{d_1-1} & d_1ay_4^{d_1-1}\\
d_2bm_2y_1^{d_2-1} & -d_2b(m_2+h_2)y_2^{d_2-1} & -d_2bm_2'y_3^{d_2-1} & d_2b(m_2'+h_2)y_4^{d_2-1}
\end{bmatrix}
\end{equation*}
We need to show that it  has full rank, as a function of $ (y_1, y_2, y_3, y_4)$, at any point in $\{G_{d_1}=H_{d_2}=0\}\setminus \{0\}$, 
    provided $ (m_2, m_2')$ are not in $D_h$,  a generalized diagonal set in $ \mathbb F _p ^2 $.  

When $J$ has rank less than $2$, assuming $y_1y_2y_3y_4\neq 0$, we can solve for each $y_i$ and plug in $G_{d_1}=0$ to get the equation
\begin{equation} \label{equation for d1 less than d2 case}
\left(\frac{1}{m_2}\right)^{\frac{d_1}{d_2-d_1}}-\left(\frac{1}{m_2+h_2}\right)^{\frac{d_1}{d_2-d_1}}-\left(\frac{1}{m_2'}\right)^{\frac{d_1}{d_2-d_1}}+\left(\frac{1}{m_2'+h_2}\right)^{\frac{d_1}{d_2-d_1}}=0
\end{equation}
If one or two of the four variables $y_1,y_2,y_3,y_4$ are zero, then a variant equation can be obtained by deleting the corresponding term(s) in the above equation.  The term $ h_2$ is fixed.  The solutions  $ (m_2, m_2')$ to \eqref{equation for d1 less than d2 case} and its variants lie in a generalized diagonal set $ D_h$. So we can apply Theorem \ref{thm: Katz} for pairs $(m_2, m_2')$ outside this set.

\smallskip 
In our main theorems, we exclude the case of the polynomials having equal degree, unless the degree is two. 
We begin with the general case of equal degree.  
Consider the case $d_1=d_2=d$. The homogeneous leading term of $G$ and $H$ are
\begin{equation*}
\begin{split}
G_{d}(y_1,y_2,y_3,y_4):=ay_1^{d}-ay_2^{d}-ay_3^{d}+ay_4^{d},
\end{split}
\end{equation*}
and
\begin{equation*}
H_{d}(y_1,y_2,y_3,y_4):=byz_1^{d}-(b(m_2+h_2)-ah_1)y_2^{d}-bm_2'y_3^{d}+(b(m_2'+h_2)-ah_1)y_4^{d},
\end{equation*}
resp. The Jacobian matrix becomes
\begin{align*}
J & :=
\begin{bmatrix}
\nabla G_{d}\\ \nabla H_{d}
\end{bmatrix}
\\ & =
\begin{bmatrix}
day_1^{d-1} & -day_2^{d-1} & -day_3^{d-1} & day_4^{d-1}\\
dbm_2y_1^{d-1} & -d(b(m_2+h_2)-ah_1)y_2^{d-1} & -dbm_2'y_3^{d-1} & d(b(m_2'+h_2)-ah_1)y_4^{d-1}
\end{bmatrix}
\end{align*}
When $y_1y_2y_3y_4\neq 0$, $J$ has rank $1$ only when
\begin{equation} \label{equation for d1=d2}
bm_2=b(m_2+h_2)-ah_1=bm_2'=b(m_2'+h_2)-ah_1.
\end{equation}
One or two terms in the above equation can be dropped if the corresponding variable is zero.  
Make the additional assumption that 
\begin{align}\label{eq:assume}
ah_1\neq bh_2,
\end{align} 
It is then easy to see that the solutions to \eqref{equation for d1=d2} and its variants form a generalized diagonal set. So Theorem \ref{thm: Katz} applies in most cases, and we are done. 

\smallskip 

In general we don't know how to deal with the case when $ah_1=bh_2$. But when $d=2$, we can use Gauss sum to prove the desired result. 

Indeed, let $\phi_i(y)=a_iy^2+b_iy$, with $a_1, a_2\neq 0$ be linearly independent. 
Let $h$ be such that $a_1h_1=a_2h_2$, then we have $b_1h_1-b_2h_2\neq 0$.
Then we have the Gauss sum
\begin{align*}
K(n_1,m_2)=
\begin{cases}
p^{-\frac{1}{2}} e_p\left(\frac{b_1n_1+b_2m_2}{4(a_1n_1+a_2m_2)}\right)  &  a_1n_1+a_2m_2\neq 0,
\\
0,  &   a_1n_1+a_2m_2=0.
\end{cases}
\end{align*}
Next denoting 
\begin{align*}
\begin{cases}
\alpha:=a_1n_1+a_2m_2,\\
\alpha':=a_1n_1+a_2m_2'\\
\beta:=b_1n_1+b_2m_2\\
\beta':=b_1n_1+b_2m_2'\\
\gamma:=-b_1h_1+b_2h_2
\end{cases}
\end{align*}
We have
\begin{align*}
&\left|\sum_{n_1}(\Delta_{(-h_1,h_2)}\tilde{K}_{\Gamma})(n_1,m_2)
\overline{(\Delta_{(-h_1,h_2)}\tilde{K}_{\Gamma})(n_1,m_2')}\right|\\
=&\frac{1}{p^2} \left|\sum_{\substack{n_1\\ \alpha\neq 0\\ \alpha'\neq 0}} e_p\left(\frac{\beta}{4\alpha}-\frac{\beta+\gamma}{4\alpha}-\frac{\beta'}{4\alpha'}+\frac{\beta'+\gamma}{4\alpha'}\right)\right|\\
=&\frac{1}{p^2} \left|\sum_{\substack{n_1\\ \alpha\neq 0\\ \alpha'\neq 0}} e_p\left(\frac{a_2(m_2-m_2')(-b_1h_1+b_2h_2)}{4(a_1n_1+a_2m_2)(a_1n_1+a_2m_2')} \right)\right|\leq 3p^{-\frac{3}{2}},
\end{align*}
where we used the following estimate by Bombieri \cite{MR200267} as long as $m_2\neq m_2'$.
\begin{proposition}
Let $f_1, f_2\in \mathbb Z[X]$, $(f_1,f_2)=1$ and $\tilde{f}_1, \tilde{f}_2\in \F_p[X]$ the corresponding polynomials over $\F_p$, 
$\tilde{f}(x)=\frac{\tilde{f}_1(x)}{\tilde{f}_2(x)}$, where $x$ is to take only values with $p\nmid f_2(x)$. Define
$$S(\tilde{f})=\sum_{x} e_p(\tilde{f}(x)).$$
Then, assuming $\deg(\tilde{f})=\deg(\tilde{f}_1)+\deg(\tilde{f}_2)\geq 1$, we have
$$|S(\tilde{f})|\leq (n-2+\deg(\tilde{f})_{\infty}) p^{\frac{1}{2}}+1$$
with $n=$ the number of the poles and $(\tilde{f})_{\infty}$ the divisor of the poles of $\tilde{f}$ over the algebraic closure $\overline{\F}_p$ (including $\infty$ if necessary).
\end{proposition}
Note that in our case, $n=2$ and $\deg(\tilde{f})_{\infty}=2$. 
\qed



\end{document}